\documentclass[12pt, oneside]{article}   	
\usepackage{geometry}                		
\usepackage[parfill]{parskip}    		
\usepackage{graphicx}				
\usepackage{amssymb}
\usepackage{amsfonts}
\usepackage{amsthm}
\usepackage{amsmath}
\usepackage{fontenc}
\usepackage[all]{xy}
\usepackage{latexsym}
\usepackage{layout}
\usepackage{epsfig}
\usepackage{graphicx}
\def\fn{\lfloor\frac{n}{2}\rfloor}
\def\fnt{\lfloor\frac{n-2}{2}\rfloor}

\def\mgt{m_{T,\gamma}}

\usepackage{tikz}
\usepackage{relsize}
\usepackage{mathabx}

\newcommand{\trpart[1]}{\mycirc_{#1}}

\newtheorem{thm}{Theorem}[section]

\newtheorem{lem}[thm]{Lemma}

\newtheorem{prp}[thm]{Proposition}

\title{Maximum Number of Minimum Dominating and Minimum Total Dominating Sets}
\author{Anant Godbole\\ East Tennessee State University\and Jessie D. Jamieson, William Jamieson\\ 
University of Nebraska, Lincoln}
\date{}							

\begin{document}
\maketitle
\abstract{Given a connected graph with domination (or total domination) number $\gamma\ge2$, we ask for the maximum number $m_\gamma$ and $\mgt$ of dominating and total dominating sets of size $\gamma$.  An exact answer is provided for $\gamma=2$ and lower bounds are given for $m_\gamma,\mgt;\gamma\ge3$.}


\section{Introduction}  We study the maximum number $m_\gamma$ and $\mgt$ of dominating or total dominating sets of mimimum size $\gamma$ in a graph $G$ with $n$ vertices.  The problem is solved for $\gamma=2$ and lower bounds are provided for $\gamma\ge3$.  Among the questions we hope to understand is the mixed  interplay of our problem with the number of edges, an increase in which aids in the creation of dominating sets of a given size, but which decreases the domination number.

\section{$\gamma=2$}  Let $G=(V,E)$ be a graph with $\vert V\vert=n\ge 3$.  A {\it dominating set} is a collection of vertices $U\subseteq V$ such that each $x\in V\setminus U$ is adjacent to some $y\in U$.  A {\it total dominating set} is a collection $U\subseteq V$ of vertices such that each $x\in V$ is adjacent to some $y\in U$.  The {\it (total) domination number} is the cardinality of the (total) dominating sets of smallest cardinality, the so-called minimum (total) dominating sets.  We claim that the maximum number of minimum dominating sets, and maximum number of minimum total dominating sets of size 2 are‎ approximately ${n\choose 2}$ and $\frac{n(n-2)}{2}$ respectively (exact numbers below).  To exhibit a lower bound on $m_{T,2}$, we first assume that $n$ is even and consider the complete multipartite graph $K_{2,2,\ldots,2}$ with $n/2$ parts of size 2 each; note that this is the same as the complete graph $K_{n}$ minus a one factor.   A dominating set cannot be of size 1, and a total dominating set of size 2 can be chosen in 4 ways for any choice of 2 parts in the bipartition.  Thus
\[m_{T,2}\ge 4{{n/2}\choose 2}=4\cdot\frac{n}{2}\frac{(n-2)}{2}\cdot\frac{1}{2}=\frac{n(n-2)}{2}.\]  This is also equal to the number of edges.  If $n$ is odd, we start with $\lfloor\frac{n}{2}\rfloor$ parts of sizes $3,2,2,\ldots,2$ in the complete bipartite graph, {\it and add a single edge in the part of size 3, though this has no effect on the number of total dominating sets}, so the number of total dominating sets is
\[4{{\fnt}\choose 2}+6(\fn-1)=\frac{n(n-2)-3}{2},\]
one less than the number of edges.  Notice that in both the even and odd constructions, the number of edges in the graph is equal to the maximum possible for a graph with (regular) domination number 2; see \cite{vizing}, \cite{goddard}.  
\begin{thm} The maximum number $m_{T,2}(n)$ of total dominating sets of size 2 in a graph $G=(V,E)$ with $\vert V\vert=n$ and total domination number 2 is given by
$$m_{T,2}(n)=\begin{cases} \frac{n(n-2)}{2}, & \mbox{if\ $n$ is even } \\ \frac{n(n-2)-3}{2} & \mbox{if\ $n$\ is\ odd } \end{cases}$$
\end{thm}
\begin{proof}  In the even case, for each total dominating set $\{a,b\}$ of size 2, we associate the edge $e=\{a,b\}$, which exists since domination is total and $\gamma=2$.  Thus, if the number of total dominating sets of size 2 were to be larger than stated, there would be a larger number of edges in the graph than the maximum allowable for the graph to have domination number two.  Thus the domination number would be one, and trivially the total domination number would equal one too, contradicting our assumptions.  If $n$ is odd, and the number of total dominating sets is at least $\frac{n(n-2)-1}{2}$, then the same argument implies that the number of edges that correspond to total dominating sets is also at least, and thus equal to, the maximum value of $\frac{n(n-2)-1}{2}$ -- but this is not as yet a contradiction.  So we recognize that this number of ``induced" edges implies that all edges are between the two vertices in a total dominating set. Now the degree of any vertex is at most $n-2$, and the sum of the degrees is $n(n-2)-1$, implying that all vertices are of degree $n-2$ except for a single vertex of degree $n-3$.  The two vertices that are not adjacent to this vertex have an edge between them, since their degree is $n-2$.  However neither of them dominate the vertex with degree $n-3$.  This contradiction completes the proof.\hfill
\end{proof}

\medskip

\noindent REMARK:
Interestingly, for higher values of $\gamma$, the maximum number of edges in a graph with domination number $\gamma$ is obtained \cite{vizing} by constructing a dominating set with $\gamma-2$ isolated vertices.  This forces the number of dominating sets to be necessarily quadratic in $n$.  This will not help us in our quest, since we wish to prove that $m_\gamma=\mgt=\Omega(n^\gamma)$.  Viewed another way, we will not be able to {easily} exploit the link between the maximum size of the graph with given domination number and the maximum number of 
dominating sets when $\gamma\ge3$.  What if isolated vertices are not allowed?  In this case, we cite the important paper of  Sanchis \cite{sanchis}, and mention that here too we run into the problem of the number of edges being too large -- causing the domination number to be barely $\gamma$, and the number of dominating sets of size $\gamma$ to be smaller than what might occur otherwise.

\medskip

\noindent OPEN PROBLEM:  Understand more precisely why a ``medium number" of edges appears to lead to the maximum number of dominating sets of minimum size.

\medskip

\noindent The above two paragraphs notwithstanding, we now return to the case $\gamma=2$ where the link between maximum size and the maximum number of 2-dominating sets can be exploited to full effect.

\begin{prp}  $m_2(n)={n\choose 2}$ when $n$ is even and $m_2(n)={n\choose 2}-1$ when $n$ is odd.
\end{prp}
\begin{proof}  For $n$ even, the example is $K_{2,2,\ldots,2}$.  If $n$ is odd, we consider $K_{3,2,\ldots,2}$ and add a single edge in the first part to get a total of ${n\choose 2}-1$ dominating sets of size 2.  All that remains to prove, therefore, is that all 2-sets cannot dominate when $n$ is odd.  Assume to the contrary that $m_2(n)=\frac{n(n-1)}{2}$.  There are thus at least ${n\choose 2}-\frac{n(n-2)-1}{2}=\frac{n+1}{2}$ edges missing.  If these absent edges align themselves into a configuration exemplified by the extremal example above, i.e. $\frac{n-3}{2}$ pairs of vertex-disjoint  missing edges together  with three vertices that have a single edge among them, then we have ${n\choose 2}-1$ dominating sets as in the example.  If not, we must have one of two configurations:  (i) Either there are three independent vertices, in which case there are at least three sets that are not dominating; or (ii) There are two sets of $K_{1,2}$ complements, and thus (allowing for overlaps) at least one non-dominating set.  This completes the proof.\hfill
\end{proof}

\section{Lower Bound Constructions for $\gamma\ge3$}

The notation in this section is deliberately different.  Recall that $m_x=m_{n,x}$ is the maximum number of dominating sets of size $x$ in a graph with domination number $x$.  Our goal is to find lower bounds on $m_{n,x}$ through construction, and we will, accordingly, exhibit graphs $G_{n,x}$ with order $n$, $\gamma(G_{n,x})=x$, and with $\Omega(n^x)$ distinct dominating sets of size $x$.  Note that, as will be reported in a paper by a different author team \cite{bill}, for very large $n$ there {\it does} exist a graph $G_{n,x}$ with $n$ vertices and $\gamma(G_{n,x})=x$, and where the number of distinct dominating sets is ${n\choose x}(1-o(1))$ -- which is  close to the maximum possible. However, this graph can only be shown to exist using probabilistic methods.  Obviously improving on the construction in this paper is an immediate area for further work.  To construct $G_{n,x}$ we will create a series of components $H_1, H_2, \dots H_k$ such that (i) the sum of the orders of the $H_i$ is $n$; and (ii) $\sum_{i=1}^k\gamma(H_i)=x$.   Finally, we will set $G_{n,x}=\bigcup_{i=1}^k H_i$.  

We will give a way to construct $G_{n,x}$ for each $x$. For $x=1$ one can simply take $G_{n,x}$ to be the complete graph $K_n$. For $x=2$ one can take $G_{n,x}$ to be the graph given in Section 2.  For $x \ge 3$ we will need more than a single component. For components $H_i$ with $r$ vertices and $\gamma(H_i)=1$ we will use $K_r$, and for components with $\gamma(H_i)=2$ we will use the graphs from Section 2 of order $r$ which we will denote $\trpart[r]$.  The next result shows that using components with $\gamma(H_i)\ge3$  will not be necessary.

\begin{lem}\label{lem:12}
	We need only decompose the vertex set into $H_i$ with $\gamma(H_i)=1$ or $\gamma(H_i)=2$, $1\le i\le k$.
\end{lem}
	\begin{proof}  We begin by considering the case $\gamma(G)=3$.  Throughout we assume that all relevant  quantities are integers.  To construct $G$ we split the vertices  into components of order $An$ and $(1-A)n$
	on which we place the graphs $K_{An}$ and $\trpart[(1-A)n]$ respectively.  This yields $An\cdot{{(1-A)n}\choose{2}}$ dominating sets of size 3, a quantity that is maximized when $A=1/3$, and which yields	$\frac{2}{27}n^3(1+o(1))$ 3-dominating sets, constituting, asymptotically, a fraction 4/9 of the ${n\choose 3}$ possible triples.  { This is our baseline construction for $\gamma=3$ and gives the largest number of  $\gamma(G)$ sets given the component method -- and using two components.}    

When $\gamma(G)=4$, we have a choice between two components each with domination number 2, or, by the $\gamma(G)=3$ case, three components with domination numbers 1, 1, and 2.  Isoperimetric considerations (or Calculus!) suggest that the component sizes be $n/2, n/2$ and $n/4, n/4, n/2$ in these two cases.  The better solution is the first since it yields ${{n/2}\choose{2}}^2\sim n^4/64$ 4-dominating sets, as opposed to the $\frac{n}{4}\frac{n}{4}{{n/2}\choose{2}}\approx \frac{n^4}{128}$ sets given by the second scenario.

Continuing in this fashion, we thus see that all partitions are into components with domination number 1 or 2.
	\end{proof}
	
\begin{lem}\label{lem:11}
	The optimum partition of $x$ will be the partition with at most one component with domination number $1$. 
\end{lem}
	\begin{proof}
		Suppose that the optimum partition has more than one $1$. Thus our partition has at least two $1$s. These two parts will be $K_r$ and $K_{r'}$ for some $r$ and $r'$. The number of ways to select dominating sets for these two components will be $r\cdot r'$. If one uses a partition that replaces the two $1$s with a single $2$. Then the number of ways to select dominating sets for these components will be ${r + r' \choose 2}$. Note that the maximum of $r \cdot r'$ occurs when $r = r'$. However, ${r + r' \choose 2}$ remains constant no matter the values of $r$ and $r'$. Further when $r = r'$, ${2 r \choose 2} \geq r^2$. Thus replacing any pair of $1$s in our partition with a $2$ will increase the number of dominating sets. It follows that our optimum partition will have no pairs of $1$s, and hence it will only have at most one $1$.
	\end{proof}

\begin{thm}
	Given the component method, the graph $G_{n,x}$ with the maximum number of dominating sets will have a partition into components with domination number $2$ with at most one component that is a complete graph. Each component with $\gamma=2$ receives at least\footnote{when $\frac{2}{x}n$ is not an integer replace $n$ by $n'$ where $n'$ is the largest integer less than $n$ such that $\frac{2}{x}n'$ is an integer and replace $n$ by $n'$. For the remaining $n-n'$ vertices distribute these as evenly as possible to the $2$s until they have been exhausted.}  $\frac{2}{x}n$ vertices and the single possible $K_r$ is on $r=\frac{1}{x}n$ vertices. 
\end{thm}
	\begin{proof}
		By Lemmas \ref{lem:12} and \ref{lem:11} the partition of $x$ needed will consist of all $2$s with at most one $1$. We first show that all of the $2$s have the same number of vertices. To do this suppose that there are two $2$ partitions of total size $2r$ and difference $2a; r > a \geq 1$.  The number of ways to dominate these two components is ${r+a \choose 2}{r-a \choose 2} \approx \frac{1}{4}(r+a)^2(r-a)^2$. If we consider these two partitions of equal size then we have ${ r \choose 2}^2 \approx \frac{1}{4}r^4$. Note that $(r+a)^2(r-a)^2 < r^4$ whenever $\frac{a}{\sqrt{2}} < r$, which is always true since $r \geq a$. Thus no components with domination number $2$ in the partition of $x$ may receive a different number of vertices (assuming $n$ can be divided into equal parts for each of the components). 
		
		If $x$ is even then we must have $\frac{2}{x}n$ vertices in each component of $G_{n,x}$. Thus let $x$ be odd, say $x=2p+1$. We know that there must be exactly $p$ $\gamma=2$ components and one $\gamma=1$ component. Let $r$ represent the number of vertices in each $2$ component. Thus there are $n-p \cdot r$ vertices in the $\gamma=1$ component. There are $(n-p \cdot r){r \choose 2}^p \approx (n-p \cdot r)\frac{r^{2p}}{2^p}$ ways to dominate $G_{n,x}$. The maximum of this function with respect to $r$ is when $r= \frac{2}{2p+1}=\frac{2}{x}$. Thus the $\gamma=1$ part has size $n-p \cdot \frac{2}{2p+1} = \frac{1}{x}n$. \hfill\end{proof}
		

\section{Efficiency of the Construction}  The number of subsets of size $x$ is ${n\choose x}\sim\frac{n^x}{x!}$.  By contrast, our construction yields $(\frac{\sqrt{2}}{x})^xn^x$ dominating sets when $x$ is even and $\frac{2^{(x-1)/2}}{x^x}n^x$ when $x$ is odd.  Since $x!\sim\sqrt{2\pi x}(x/e)^x$, we see that asymptotically a fraction $(\sqrt{2}/e)^x$ of the $x$-sets dominate.  It would be interesting to see how close to ${n\choose x}$ we can make the number of dominating sets by constructive means.

\section{Acknowledgments} The research of AG was sponsored by NSF Grant 1004624.

\end{document}